\documentclass{amsart}
\usepackage{amssymb,amsmath,amsthm,latexsym,stmaryrd,xcolor}

\newtheorem{theorem}{Theorem}[section]

\newcommand{\ie}{i.\hspace{.5pt}e.\ }
\newcommand{\f}{\varphi}
\newcommand{\g}{\tilde{g}}
\newcommand{\bg}{\bar{g}}
\newcommand{\n}{\nabla}

\newcommand{\M}{(\mathcal{M},\A\f,\A\xi,\A\eta,\A{}g)}
\newcommand{\MM}{\mathcal{M}}

\newcommand{\I}{\iota}

\newcommand{\R}{\mathbb R}

\newcommand{\X}{\mathfrak X}
\newcommand{\F}{\mathcal{F}}
\newcommand{\LL}{\mathcal{L}}
\newcommand{\ta}{\theta}
\newcommand{\om}{\omega}
\newcommand{\lm}{\lambda}
\newcommand{\sm}{\sigma}
\newcommand{\gm}{\gamma}
\newcommand{\al}{\alpha}
\newcommand{\bt}{\beta}

\newcommand{\D}{\mathrm{d}}
\DeclareMathOperator{\Span}{span} % the span
\DeclareMathOperator{\im}{im} % the image of linear map

\newcommand{\ddu}[1]{\dfrac{\partial u}{\partial x^{#1}}}
\newcommand{\ddv}[1]{\dfrac{\partial v}{\partial x^{#1}}}

\newcommand{\p}{\partial}

\newcommand{\thmref}[1]{Theorem~\ref{#1}}

\newcommand{\A}{\allowbreak{}}

\begin{document}

%=================================================================
% Full title of the paper (Capitalized)
\title[Yamabe Solitons and Vertical Torse-Forming  Vector Fields ...]
{Yamabe Solitons on Conformal Almost Contact Complex Riemannian Manifolds with Vertical Torse-Forming Vector Field}

\author{Mancho Manev}

\address[M. Manev]{
University of Plovdiv Paisii Hilendarski,
Faculty of Mathematics and Informatics,
Department of Algebra and Geometry,
24 Tzar Asen St.,
Plovdiv 4000, Bulgaria
\&
Medical University -- Plovdiv,
Faculty of Pharmacy,
Department of Medical Physics and Biophysics,
15A Vasil Aprilov Blvd,
Plovdiv 4002, Bulgaria}
\email{mmanev@uni-plovdiv.bg}

\begin{abstract}
A Yamabe soliton is considered on an almost contact complex Riemannian manifold (also known as an almost contact B-metric manifold) which is obtained by a contact conformal transformation of the Reeb vector field, its dual contact 1-form, the B-metric, and its associated B-metric.
The case when the potential is a torse-forming vector field of constant length on the vertical distribution determined by the Reeb vector field is studied.
In this way, manifolds from one of the main classes of the studied manifolds are obtained. The same class contains the conformally equivalent manifolds of cosymplectic manifolds by the usual conformal transformation of the given B-metric.
An explicit 5-dimensional example of a Lie group is given, which is characterized in relation to the obtained results.
\end{abstract}


 %Differential-geometric methods (tensors, connections, symplectic, Poisson, contact, Riemannian, nonholonomic, etc.) IN 70-xx Mechanics of particles and systems; 	70Gxx		General models, approaches, and methods

\maketitle

%%%%%%%%%%%%%%%%%%%%%%%%%%%%%%%%%%%%%%%%%%

\section{Introduction}

The concept of Yamabe flow was first introduced by R.{ }Hamilton in \cite{Ham88} to construct
Yamabe metrics on compact Riemannian manifolds. On a Riemannian or pseudo-Riemannian manifold, a time-dependent metric $g(t)$ is said to evolve by the Yamabe flow if $g$ satisfies the equation
\[
\frac{\p}{\p t}g=-\tau g,
\]
where $\tau(t)$ is the scalar curvature of $g(t)$.
According to mathematical physics, the Yamabe flow corresponds to the case of fast diffusion of the porous medium equation \cite{ChLuNi}.

A self-similar solution of the Yamabe flow is called a Yamabe soliton and is determined by (\cite{BarRib})
\begin{equation*}\label{YS-intro}
  \frac12 \LL_{v} g = (\tau - \sm) g,
\end{equation*}
where $\LL_{v} g$ is the Lie derivative of $g$ along a vector field $v$ and $\sm$ is a constant.
Various authors have studied Yamabe solitons on manifolds equipped with different types of structures   (see e.g. \cite{Cao1,Chen1,Das1,Gho1,HuiMan18,Roy1}).

In \cite{Man73}, we began the study of Yamabe solitons on \emph{almost contact B-metric manifolds}.
These manifolds are also known as \emph{almost contact complex Riemannian manifolds}, abbreviated as \emph{accR manifolds}.
The geometry of these manifolds is largely determined by the presence of a pair of B-metrics that are related each other by the almost contact structure.

Recall that the conformal class of the metric is preserved by the Yamabe flow.
With this reason in mind, we study Yamabe solitons and conformal transformations together in the case under consideration.

The so-called contactly conformal transformations were studied in \cite{Man4}.
They deformed not only the metric but also the Reeb vector field and its associated contact 1-form using the pair of B-metrics.
The partial case when this transformation changes only the B-metrics is studied in \cite{ManGri1}.

The Ganchev--Mihova--Gribachev classification of the studied manifolds given in \cite{GaMiGr} consists of
eleven basic classes. We call four of them the main classes, since
the covariant derivatives of the structure tensors with respect to the Levi-Civita connection of each of the B-metrics are expressed only by the pair of B-metrics and the corresponding traces.
Let us recall \cite{ManGri1}, where it is proved that the direct sum of the four main classes
is closed under the action of contactly conformal transformations.

The main results in the present paper concern two of the four main classes, namely $\F_1$ and $\F_5$, as well as their intersection $\F_0$. % containing the cosymplectic type of the considered manifolds.
The first of them is defined in \cite{GaMiGr} as an analogue of the unique main class $\mathcal{W}_1$ of almost complex manifolds with Norden metric according to the Ganchev--Borisov classification in \cite{GaBo}.
Furthermore, it is known from \cite{ManGri1} that any $\F_0$-manifold can be transformed into an $\F_1$-manifold by a usual conformal transformation of the B-metric.
An example of an $\F_5$-manifold as an isotropic hypersurface with respect to the associated B-metric in an even-dimensional real space is given in \cite{GaMiGr} and it is noted that the class $\F_5$ is analogous in some sense to the class of $\al$-Sasakian manifolds (we can add---to the class of $\bt$-Kenmotsu manifolds) in the theory of almost contact metric manifolds.
Results on the geometric properties of manifolds from these main classes, as well as examples of them, can be found e.g. in \cite{HMan}--\cite{ManIv38}.

In \cite{Man73} we introduced Yamabe solitons on accR manifolds and started their study for two of the simplest types of these manifolds, namely cosymplectic and Sasaki-like.
We found that the resulting manifolds in both cases belong to one of the main classes, the only one that contains the conformally equivalent manifolds of the cosymplectic ones by the usual conformal transformations.

The aim of this paper is to study another possibility for the initial manifold, which is determined by a natural condition, also intensively studied in relation with solitons, namely the use of a torse-forming vector field.

The structure of the present paper is as follows.
Section 1 is the present introduction.
Section 2 recalls the basic facts for the investigated manifolds, the relevant transformations of the structure tensors on them, and the notion of the Yamabe soliton
on a transformed accR manifold.
Section 3 is devoted to the investigation of the case described in the title of the paper.
Section 4 supports the main theorem by providing an explicit example of a hypersurface as a manifold equipped with the investigated structures and having arbitrary dimension.

\section{Preliminaries}%Almost contact B-metric manifolds, contact conformal transformations and Yamabe solitons}

\subsection{Almost contact complex Riemannian manifolds}

Let $\M$ be an almost contact manifold with
B-metric, also known as an \emph{almost contact complex Riemannian manifold} (abbreviated accR manifold)  (see e.g. \cite{GaMiGr,IvMaMa45}).
This means that $\MM$ is
a $(2n+1)$-dimensional differen\-tia\-ble manifold with an almost
contact structure $(\f,\xi,\eta)$
and a pseudo-Riemannian metric
$g$  of signature $(n+1,n)$ such that  \cite{GaMiGr}
\begin{gather}
\f\xi = 0,\quad \f^2 = -\I + \eta \otimes \xi,\quad
\eta\circ\f=0,\quad \eta(\xi)=1, \label{str}\\%
g(X, Y ) = - g(\f X, \f Y ) + \eta(X)\eta(Y) \label{g}
\end{gather}
for arbitrary $X$, $Y$ of the algebra $\X(\MM)$ on the smooth vector
fields on $\MM$, where $\I$ denotes the identity on $\X(\MM)$.

From now on, we denote by $X$, $Y$, $Z$ arbitrary elements of
$\X(\MM)$ or vectors in the tangent space $T_p\MM$ of $\MM$ at an arbitrary
point $p$ in $\MM$.

The given B-metric has its associated B-metric $\tilde{g}$ on $\M$ defined by
\begin{equation*}\label{tg}
\tilde{g}(X,Y)=g(X,\f Y)\allowbreak+\eta(X)\eta(Y).
\end{equation*}

A classification of the accR manifolds
is given in \cite{GaMiGr} by Ganchev, Mihova and Gribachev. This classification
includes eleven basic classes $\F_1$, $\F_2$, $\dots$, $\F_{11}$.
Their intersection is the special class $\F_0$ defined by
condition for the vanishing of $F$. Hence $\F_0$ is the class of cosymplectic accR manifolds,
where the structures $\f$, $\xi$, $\eta$, $g$, $\tilde{g}$ are $\n$-parallel.

This classification is made in terms of the tensor $F$ of type (0,3) defined on $\M$ by
\begin{equation*}\label{F=nfi}
F(X,Y,Z)=g\bigl( \left( \nabla_X \f \right)Y,Z\bigr),
\end{equation*}
where $\n$ is the Levi-Civita connection of $g$.
The following properties of $F$ are consequences of \eqref{str} and \eqref{g}:
\begin{equation*}\label{F-prop}
%\begin{split}
F(X,Y,Z)=F(X,Z,Y)
=F(X,\f Y,\f Z)+\eta(Y)F(X,\xi,Z) +\eta(Z)F(X,Y,\xi)
%\end{split}
\end{equation*}
and relations of $F$ with $\n\xi$ and $\n\eta$ are:
\begin{equation}\label{FXieta}
    \left(\n_X\eta\right)Y=g\left(\n_X\xi,Y\right)=F(X,\f Y,\xi).
\end{equation}
The following 1-forms are associated with $F$ and known as Lee forms of the manifold:
\begin{equation}\label{tataom}
    \ta=g^{ij}F(E_i,E_j,\cdot),\qquad \ta^*=g^{ij}F(E_i,\f E_j,\cdot),\qquad \om=F(\xi,\xi,\cdot),
\end{equation}
where $g^{ij}$ are the contravariant components of $g$
with respect to a basis $\left\{E_i;\xi\right\}$
$(i=1,2,\dots,2n)$ of $T_p\MM$.
These Lee forms satisfy the following general identities \cite{ManGri1}
\begin{equation}\label{tataom=id}
    \ta^*\circ \f=-\ta\circ\f^2,\qquad \om(\xi)=0.
\end{equation}

\subsection{Conformal transformations of the accR structure}

The author and K.\ Gribachev introduced the so-called \emph{contactly conformal transformation of the B-metric} in \cite{ManGri1},
using the pair of B-metrics $g$ and $\g$, as well as $\eta\otimes\eta$.
This transformation deforms $g$ into a new B-metric $\bar g$ for $\M$.
This transformation is later generalized in \cite{Man4} to a \emph{contact conformal transformation of the accR structure}, which yields a new almost contact structure with B-metric $(\f,\bar\xi,\bar \eta, \bar g)$ as follows
\begin{equation}\label{cct}
\begin{array}{l}
\bar\xi=e^{-w}\xi,\qquad \bar\eta=e^{w}\eta,\\[4pt]
\bar g= e^{2u}\cos{2v}\, g + e^{2u}\sin{2v}\, \g + \left(e^{2w}-e^{2u}\cos{2v}-e^{2u}\sin{2v}\right)\eta\otimes\eta,
\end{array}
\end{equation}
where $u, v, w$ are differentiable functions on $\MM$. Hereafter, we call a transformation of this type an \emph{accR transformation} for short.

Expression of the corresponding tensor $\bar F$ by $F$ under accP transformation is given in \cite{Man4}, see also \cite[(22)]{ManIv38}.
%\begin{subequations}\label{ff}
%\begin{equation}
%\begin{aligned}
%    2\bar{F}(X,Y,Z)=&\ 2e^{2u}\cos{2v}\, F(X,Y,Z)
%    \\[4pt]
%    &
%    +e^{2u}\sin{2v} \left[F(\f Y,Z,X)-F(Y,\f Z,X)+F(X,\f Y,\xi)\eta(Z)\right]\\[4pt]
%    &+e^{2u}\sin{2v}
%    \left[F(\f Z,Y,X)-F(Z,\f Y,X)+F(X,\f
%    Z,\xi)\eta(Y)\right]\\[4pt]
%    &+(e^{2w}-e^{2u}\cos{2v})\left[F(X,Y,\xi)+F(\f Y,\f
%    X,\xi)\right]\eta(Z)\\[4pt]
%    &+(e^{2w}-e^{2u}\cos{2v})
%    \left[F(X,Z,\xi)+F(\f Z,\f X,\xi)\right]\eta(Y)\\[4pt]
%    &+(e^{2w}-e^{2u}\cos{2v})
%    \left[F(Y,Z,\xi)+F(\f Z,\f Y,\xi)\right]\eta(X)\\[4pt]
%    &+(e^{2w}-e^{2u}\cos{2v})\left[F(Z,Y,\xi)+F(\f Y,\f Z,\xi)\right]\eta(X)
%\\[4pt]
%    &-2e^{2u}\left[
%    \cos{2v}\,\al(Z)
%    +\sin{2v}\,\bt(Z)\right]g(\f X,\f Y)
%\\[4pt]
%    &-2e^{2u}\left[
%    \cos{2v}\,\al(Y)
%    +\sin{2v}\,\bt(Y)\right]g(\f X,\f Z)
%%\\[4pt]
%\end{aligned}
%\end{equation}
%\begin{equation}
%\begin{aligned}
%    &-2e^{2u}\left[
%    \cos{2v}\,\bt(Z)
%    -\sin{2v}\,\al(Z)\right]g(X,\f Y)
%\\[4pt]
%    &-2e^{2u}\left[
%    \cos{2v}\,\bt(Y)
%    -\sin{2v}\,\al(Y)\right]g(X,\f Z)
%\\[4pt]
%    &+2e^{2w}\eta(X)\left[\eta(Y)\D w(\f Z)+\eta(Z)\D w(\f
%    Y)\right],
%\end{aligned}
%\end{equation}
%\end{subequations}
Moreover, the relations between the Lee forms of the manifolds $\M$ and $(\MM,\f,\bar\xi,\bar\eta,\bar g)$ are the following %(see \cite{Man4})
\begin{equation}\label{ttbartt}
%\begin{array}{l}
    \bar{\ta} = \ta+2n\, \al,\qquad
%    \\[4pt]
    \bar{\ta}^* = \ta^* +2n\,\bt,\qquad
%    \\[4pt]
    \bar{\om}  = \om +\D w\circ \f,
%\end{array}
\end{equation}
where we use the following notations for  brevity
\begin{equation}\label{albt}
\al = \D u\circ \f + \D v, \qquad \bt = \D u - \D v\circ \f.
\end{equation}

The following mutually equivalent properties follow immediately from \eqref{albt}
\begin{equation}\label{albt2}
\al\circ \f^2 + \bt\circ \f = \al\circ \f - \bt\circ \f^2 = 0.
\end{equation}

It is well known the following expression of the Lie derivative in terms of the covariant derivative with respect to the Levi-Civita connection $\bar\n$ of %the B-metric
$\bg$
\begin{equation}\label{L1}
  \left(\LL_{\bar\xi} \bg\right)(X,Y) = \bg\left(\bar\n_X \bar\xi, Y\right)+\bg\left(X, \bar\n_Y \bar\xi\right).
\end{equation}

\section{Main results}%Yamabe solitons, contact conformal transformations and torse-forming vector fields}

In \cite{Man73}, it is said
that the contact transformed B-metric $\bg$ generates a Yamabe soliton with potential the Reeb vector field $\bar\xi$ and soliton constant $\sm$ on a conformal accR manifold $(\MM,\f,\bar\xi,\bar\eta,\bg)$, if the following condition is satisfied
\begin{equation}\label{YS}
  \frac12 \LL_{\bar\xi} \bg = (\bar\tau - \sm)\bar g,
\end{equation}
where $\bar\tau$ is the scalar curvature of $\bar g$.

A vector field $\vartheta$ on a (pseudo-)Riemannian manifold $(\MM,g)$ is called \emph{torse-form\-ing vector field} if it
satisfies the following condition for arbitrary vector field $x\in \X(\MM)$ % (where $TM$ is the tangent bundle of $\MM$):
\begin{equation}\label{tf-v}
	\n_x \vartheta = f\,x + \gm(x)\vartheta,
\end{equation}
where $f$ is a differentiable function and $\gm$ is a 1-form \cite{Sch54,Yano44}.
The 1-form $\gm$ is called the \emph{generating form} and
the function $f$ is called the \emph{conformal scalar} of $\vartheta$ \cite{MihMih13}.

Further, we consider a torse-forming vector field $\vartheta$ on an accR manifold $\M$, \ie \eqref{tf-v} is valid.

In addition, those vector fields that have a special arrangement regarding the structure under consideration are naturally distinguished.
The almost contact structure on $\MM$ gives rise to two mutually orthogonal distributions with respect to the B-metric $g$, namely the contact (or horizontal) distribution
$\mathcal{H}=\ker(\eta)=\im(\f)$ and the vertical distribution $\mathcal{H}^\bot=\Span(\xi)=\ker(\f)$.

For these reasons, we study the case where the torse-forming vector field $\vartheta$ is vertical, \ie $\vartheta\in H^\bot=\Span(\xi)$.
Therefore,
$\vartheta$ is collinear to $\xi$, \ie the following equality holds
\begin{equation}\label{vert}
    \vartheta=k\,\xi,
\end{equation}
where $k$ is a nonzero function on $\MM$ and obviously
$k=\eta(\vartheta)$ holds true.

By virtue of \eqref{tf-v} and  \eqref{vert}, we obtain
\begin{equation}\label{dkxxi}
    \D k(x)\,\xi + k\n_x \xi = fx+k\gm(x)\,\xi,
\end{equation}
which after applying $\eta$ gives
\begin{equation}\label{dk}
    \D k(x) = f\eta(x)+k\gm(x).
\end{equation}
Thus, \eqref{str}, \eqref{dkxxi} and  \eqref{dk} imply
\begin{equation*}\label{nxi}
    \n_x \xi = -\frac{f}{k}\,\f^2x.
\end{equation*}

The last formula due to \eqref{FXieta} gives the following
\[
\left(\n_x\eta\right)(y) = -\frac{f}{k}\,g(\f x,\f y),
\]
\begin{equation}\label{Fxyxi}
    F(x,y,\xi) = -\frac{f}{k}\,g(x,\f y).
\end{equation}

Equality \eqref{Fxyxi} shows that
the following basic properties for the operation of the structure on the considered manifold are fulfilled
\begin{equation}\label{Fxi-prop}
    F(x,y,\xi) = F(y,x,\xi) = -F(\f x,\f y,\xi)= F(\f^2 x,\f^2 y,\xi).
\end{equation}

Using \eqref{tataom}, we obtain the following results for the Lee forms
\begin{equation}\label{Lee}
    \ta(\xi) = 0,\qquad \ta^*(\xi) = 2n\,\frac{f}{k},\qquad \om=0.
\end{equation}
Then, by virtue of \eqref{str} and \eqref{Lee}, the Lee forms of the accR manifold with a vertical torse-forming vector field $\vartheta$ satisfy the following
\begin{equation}\label{Lee-tf}
    \ta = -\ta\circ\f^2,\qquad \ta^* = -\ta^*\circ\f^2 + 2n\,\frac{f}{k}\eta,\qquad \om=0.
\end{equation}

In \cite{HMan}, it is proved that
$\M$ belongs to a certain class $\F_i$ $(i=1,\dots,11)$ from the Ganchev--Mihova--Gribachev classification if and only if
$F$ satisfies the condition $F=F^i$, where the components $F^i$ of
$F$.
The components of $F$ for the classes $\F_1$ and $\F_5$ mentioned in the main result of this paper are as follows
\begin{equation}\label{Fi-Ico-F1F5}
\begin{split}
&F^{1}(x,y,z)=\frac{1}{2n}\left\{g(\f x,\f y)\ta(\f^2 z) +g(x,\f y)\ta(\f z)\right.\\
&\phantom{F^{1}(x,y,z)=\frac{1}{2n}\left\{\right.}
\left.
+g(\f x,\f z)\ta(\f^2 y)+g(x,\f z)\ta(\f y)\right\}, \\[4pt]
&F^{5}(x,y,z)=-\frac{\ta^*(\xi)}{2n}\Bigl\{g(x,\f y)\eta(z)+g(x,\f z)\eta(y)\Bigr\}.
\end{split}
\end{equation}

In addition, it is said that an accR manifold
belongs to a direct sum of two or more basic classes, \ie
$\M\in\F_i\oplus\F_j\oplus\cdots$, if and only if the fundamental
tensor $F$ on $\M$ is the sum of the corresponding components
$F^i$, $F^j$, $\ldots$ of $F$, \ie the following condition is
satisfied $F=F^i+F^j+\cdots$ for  $i,j\in\{1,2,\dots,11\}$, $i\neq j$.

Bearing in mind \eqref{Fxyxi}, \eqref{Fxi-prop} and the expressions for the components $F^i$ from \cite{HMan}, we establish the vanishing of the components for $i\in\{4,6,7,8,9,11\}$,
%\[
%F^4=F^6=F^7=F^8=F^9=F^{11}=0,
%\]
which means that the common class of the studied accR manifolds is $\F_1\oplus\F_2\oplus\F_3\oplus\F_5\oplus\F_{10}$. Moreover, among the basic classes, only $\F_5$ can contain such manifolds.
Note that $\F_5$-manifolds are counterparts of $\bt$-Kenmotsu
manifolds in the case of almost contact metric manifolds.
%Let us recall from \cite{GaMiGr} the definition condition of this class:
%\[
%\F_{5}: \quad F(x,y,z)=-\dfrac{\ta^*(\xi)}{2n}\bigl\{g( x,\f y)\eta(z)+g(x,\f z)\eta(y)\bigr\}.
%\]

%\vspace{1in}

Let us consider $\M$ with a vertical torse-forming vector field $\vartheta$ and an accR structure such that it is an $\F_5$-manifold, \ie $F=F^5$ from \eqref{Fi-Ico-F1F5} is valid.
In this case, due to $\ta=\om=0$, $\ta^*=\ta^*(\xi)\eta$ and \eqref{Lee-tf}, we have
\begin{equation*}\label{Lee-tf-F5}
    \ta = 0,\qquad \ta^* = 2n\,\frac{f}{k}\eta,\qquad \om=0.
\end{equation*}
Furthermore, the expression of $\bar F$ through $F$ under accP transformation given in
\cite{Man4}
 takes the form
\begin{equation}\label{ff5}
\begin{aligned}
    \bar{F}(X,Y,Z)= {}-e^{2u}&\Bigl\{\lm(Z)g(\f X,\f Y)
    +\lm(Y)g(\f X,\f Z)
\\[4pt]
    &+\mu(Z)g(X,\f Y)
    +\mu(Y)g(X,\f Z)\Bigr\}
\\[4pt]
    {}+e^{2w}&\eta(X)\Bigl\{\eta(Y)\D w(\f Z)+\eta(Z)\D w(\f
    Y)\Bigr\},
\end{aligned}
\end{equation}
where the 1-forms $\lm$ and $\mu$ are expressed by $\al$ and $\bt$ introduced in \eqref{albt} as follows
\[
\begin{split}
\lm(Z)&=\cos{2v}\,\al(Z)
    +\sin{2v}\,\left\{\bt(Z)+\frac{f}{k}\eta(Z)\right\},\\
\mu(Z)&=\cos{2v}\,\left\{\bt(Z)+\frac{f}{k}\eta(Z)\right\}
    -\sin{2v}\,\al(Z).
\end{split}
\]

From \eqref{ff5}, bearing in mind \eqref{FXieta} and the first two equalities of \eqref{cct}, % for the manifolds under consideration,
%\begin{equation}\label{lemma}
%    F(X,\f Y, \xi)=\left(\n_X \eta\right)(Y)=g\left(\n_X \xi, Y\right),
%\end{equation}
we obtain
\begin{equation}\label{xi0}
\begin{aligned}
    \bar g\left(\bar\n_X \bar\xi, Y\right)=&\ e^{2u-w}\Bigl\{\lm(\xi)g( X,\f Y)-\mu(\xi)g(\f X,\f Y)\Bigr\}
\\[4pt]
    &+e^{w}\eta(X)\D w(\f^2 Y).
\end{aligned}
\end{equation}

Using the last equality of \eqref{cct}, we have the following system of equalities
\begin{equation*}\label{gbarg}
\begin{array}{l}
\bar g(\f \cdot,\f \cdot)=e^{2u}\cos{2v}\,g(\f \cdot,\f \cdot)-e^{2u}\sin{2v}\,g( \cdot,\f \cdot),\\[4pt]
\bar g( \cdot,\f \cdot)=e^{2u}\cos{2v}\,g( \cdot,\f \cdot)+e^{2u}\sin{2v}\,g(\f \cdot,\f \cdot),
\end{array}
\end{equation*}
which can be expressed vice versa in the form
\begin{equation}\label{gbarg2}
\begin{array}{l}
g(\f \cdot,\f \cdot)=e^{-2u}\cos{2v}\,\bar g(\f \cdot,\f \cdot)+e^{-2u}\sin{2v}\,\bar g( \cdot,\f \cdot),\\[4pt]
g( \cdot,\f \cdot)=e^{-2u}\cos{2v}\,\bar g( \cdot,\f \cdot)-e^{-2u}\sin{2v}\,\bar g(\f \cdot,\f \cdot).
\end{array}
\end{equation}

Applying \eqref{gbarg2} in  \eqref{xi0}, we obtain
\begin{equation}\label{xi0bar}
\begin{aligned}
    \bar g\left(\bar\n_X \bar\xi, Y\right)=&\ \al(\bar\xi)\,
    \bar g( X,\f Y)
    -\left\{\bt(\bar\xi)+\frac{f}{k}e^{-w}\right\}\bar g(\f X,\f Y)
\\[4pt]
    &+\bar\eta(X)\D w(\f^2 Y).
\end{aligned}
\end{equation}

Note that according to \eqref{albt}, the following equalities are satisfied
\begin{equation}\label{albtxi}
\al(\bar\xi)=\D v(\bar\xi),\qquad \bt(\bar\xi)=\D u(\bar\xi).
\end{equation}

Combining the expression in \eqref{xi0bar} with \eqref{albt}, \eqref{L1} and \eqref{albtxi} gives the following formula
\begin{equation}\label{Lxi0}
\begin{aligned}
    \left(\LL_{\bar\xi} \bg\right)(X,Y)=&\ 2\left\{\D v(\bar\xi)
    \bar g( X,\f Y)
    -\left[\D u(\bar\xi)+\frac{f}{k}e^{-w}\right]\bar g(\f X,\f Y)\right\}
\\[4pt]
    &+\bar\eta(X)\D w(\f^2 Y)+\bar\eta(Y)\D w(\f^2 X).
\end{aligned}
\end{equation}

\begin{theorem}\label{thm:F0-YS}
An accR manifold of the main class $\F_5$ with a vertical torse-forming vector field $\vartheta$ can be transformed by an accR transformation so that the transformed B-metric is a Yamabe soliton with potential the transformed Reeb vector field
and a soliton constant $\sm$
if and only if the functions $(u,v,w)$ of the used trans\-formation satisfy the conditions
\begin{gather}\label{F5-YS-uv}
    \D u(\xi)=-\frac{f}{k},\qquad
\D v(\xi)=0,\\[4pt]
%\end{equation}
%\begin{equation}
\label{dw}
  \D w=\D w(\xi)\eta.
\end{gather}

Moreover, the obtained Yamabe soliton has a constant scalar curvature with value $\bar\tau=\sm$ and the obtained accR manifold belongs to the subclass of the main class $\F_1$ determined by the conditions:
\begin{equation}\label{F0F1-YS}
    \bar\ta=2n\left\{\D u\circ\f+\D v\right\}, \qquad
    \bar\ta^*=-2n\left\{\D u\circ\f^2 + \D v\circ\f\right\}, \qquad
    \bar\om= 0.
\end{equation}

As a corollary, if $(u,v)$ are a pair of $\f$-holomorphic functions then the transformed manifold belongs to the special class $\F_0$ of cosymplectic accR manifolds.
\end{theorem}
\begin{proof}
The condition that $\bar g$ generates a Yamabe soliton with potential $\bar \xi$ %and soliton constant $\sm$
on a conformal accR manifold $(\MM,\f,\bar\xi,\bar\eta,\bg)$ means that
\eqref{YS} is satisfied. In this case, \eqref{Lxi0} takes the form
\begin{equation}\label{Lxi00}
\begin{aligned}
    &2\left(\bar\tau-\sm\right)\left\{ -\bg(\f X,\f Y)+\bar\eta(X)\bar\eta(Y)\right\}\\[4pt]
    &= 2\left\{\D v(\bar\xi)
    \bar g( X,\f Y)
    -\left[\D u(\bar\xi)+\frac{f}{k}e^{-w}\right]\bar g(\f X,\f Y)\right\}
\\[4pt]
    &\phantom{=} +\bar\eta(X)\D w(\f^2 Y)+\bar\eta(Y)\D w(\f^2 X).
\end{aligned}
\end{equation}

The substitution $\bar\xi$ for $Y$ in \eqref{Lxi00} gives %the following
\begin{equation}\label{tsDw}
2\left(\bar\tau-\sm\right)\bar\eta(X)=\D w(\f^2 X).
\end{equation}
Then we replace $X$ with $\bar\xi$ and get an expression of the scalar curvature for $\bg$ as follows
\begin{equation}\label{tausigma}
  \bar\tau=\sigma.
\end{equation}
Hence, \eqref{tsDw} implies the vanishing of $\D w\circ\f^2 $, which is equivalent to the condition in
\eqref{dw}.

After applying \eqref{tausigma} and \eqref{dw} to \eqref{Lxi00}, we obtain
\eqref{F5-YS-uv}.

Note that $\LL_{\bar\xi} \bg$ vanishes due to \eqref{tausigma} and \eqref{YS}, \ie $\bar\xi$ is a Killing vector field in the considered case. For more results concerning Killing vector fields on Riemannian manifolds, see \cite{DesBel}.

As a next step, we substitute \eqref{gbarg2} in \eqref{ff5} and get
\begin{equation}\label{ff5=}
\begin{aligned}
    \bar{F}(X,Y,Z)= {}&-\al(Z)\bg(\f X,\f Y)-\left\{\bt(Z)+\frac{f}{k}\eta(Z)\right\}\bg(X,\f Y)
\\[4pt]
    &-\al(Y)\bg(\f X,\f Z)-\left\{\bt(Y)+\frac{f}{k}\eta(Y)\right\}\bg(X,\f Z)
\\[4pt]
    &+\bar\eta(X)\left\{\bar\eta(Y)\D{w}(\f Z)+\bar\eta(Z)\D{w}(\f Y)\right\}.
\end{aligned}
\end{equation}

Then, computing the Lee forms for $\bar F$ in \eqref{ff5=} by \eqref{tataom}, \eqref{albt} and \eqref{albt2}, we obtain the following expressions
\begin{equation}\label{tataAB}
%\begin{split}
    \al=\frac{1}{2n}\bar\ta, \qquad %\\[4pt]
    \bt=\frac{1}{2n}\bar\ta^*-\frac{f}{k}\eta, \qquad %\\[4pt]
    \D w\circ\f = \bar\om.
%\end{split}
\end{equation}
By virtue of \eqref{albtxi}, \eqref{F5-YS-uv}, \eqref{dw} and \eqref{tataAB}, we obtain
\begin{equation*}\label{tataAB=}
    \bar{\ta}(\bar\xi)=0,\qquad
    \bar{\ta}^*(\bar\xi)=0,\qquad
    \bar{\om}=0,
\end{equation*}
where the first two equalities are equivalent to the following properties, respectively
\begin{equation}\label{tataAB==}
    \bar{\ta}=-\bar{\ta}\circ\f^2,\qquad
    \bar{\ta}^*=-\bar{\ta}^*\circ\f^2.
\end{equation}
Thus, \eqref{albt}, \eqref{F5-YS-uv}, \eqref{dw}, \eqref{tataAB} and \eqref{tataAB==}
imply \eqref{F0F1-YS}.
%The results in \eqref{F0F1-YS} are also a consequence of \eqref{albt}, \eqref{ttbartt}, \eqref{Lee-tf-F5} and \eqref{dw}.

Substituting \eqref{dw}, \eqref{tataAB} and \eqref{tataAB==}
in \eqref{ff5=}, we get
\begin{equation}\label{ff0bar2}
\begin{aligned}
    \bar{F}(X,Y,Z)= \frac{1}{2n}\Bigl\{&\bg(\f X,\f Y)\,\bar\ta(\f^2 Z) + \bg(X,\f Y)\,\bar\ta^*(\f^2 Z)
\\[4pt]
    &\!%\hspace{20pt}
    +\bg(\f X,\f Z)\,\bar\ta(\f^2 Y) + \bg(X,\f Z)\,\bar\ta^*(\f^2 Y)\Bigr\}.
\end{aligned}
\end{equation}
The resulting expression of $\bar F$ in \eqref{ff0bar2} can be written in the form
\begin{equation}\label{F1}
    \bar F=\bar F^1,
\end{equation}
using the corresponding component of $\bar F$ in \eqref{Fi-Ico-F1F5} and the relation in the first equality of \eqref{tataom=id}.
The expression of $\bar F$ in \eqref{F1} means that the transformed manifold $(\MM,\f,\bar\xi,\bar\eta,\bg)$ belongs to the main class $\F_1$, according to the classification of Ganchev--Mihova--Gribachev in \cite{GaMiGr}.

The assumption that $(u,v)$ is a $\f$-holomorphic pair,  \ie $(u,v)$ satisfy the conditions
\[
\D u\circ\f=\D v\circ\f^2, \qquad
\D u\circ\f^2 =- \D v\circ\f,
\]
and \eqref{F0F1-YS} yield the vanishing of all Lee forms of the obtained $\F_1$-manifold, which means that it is in $\F_0$.
\end{proof}

As a result of \eqref{F5-YS-uv} and \eqref{dw}, we derive the following conclusion. The situation of \thmref{thm:F0-YS} occurs when the functions $(u,v,w)$ of the used accR transformation satisfy the conditions:
\begin{itemize}
  \item $v$ is a vertical constant, \ie constant on $\mathcal{H}^\bot$;
  \item $w$ is a horizontal constant, \ie constant on $\mathcal{H}$.
\end{itemize}

\section{Example}

We recall a known example of an $\F_5$-manifold given in \cite{GaMiGr} as Example~3.
The space $\mathbb R^{2n+2}=\left\{(x^1, \dots, x^{n+1}; x^{n+2}, \dots, x^{2n+2})|\ x^i\in\R\right\}$ is considered as a complex Riemannian manifold with the canonical complex structure
$J$ and the metric $G$ defined for any $X=\lm^i \frac{\p}{\p x^i}+\mu^i\frac{\p}{\p x^{n+i+1}}$ by
\[
J \frac{\p}{\p x^i} = \frac{\p}{\p x^{n+i+1}},\quad
J \frac{\p}{\p x^{n+i+1}} = -\frac{\p}{\p x^i},\quad
G(X,X)=-\delta_{ij} \lm^i\lm^j + \delta_{ij} \mu^i\mu^j,
\]
where $\delta_{ij}$ is the Kronecker delta for $i,j\in\{1,\dots,n+1\}$.

The manifold $\M$ is constructed as a hypersurface of $(\mathbb R^{2n+2},\A{}J,\A{}G)$ determined by
\[
G(Z,JZ)=0,\qquad G(Z,Z)=\cosh^2 t, \qquad  t>0,
\]
where $Z$ is the position vector of a point in $\mathbb R^{2n+2}$.

The Reeb vector field is defined by $\xi=-JN$, where $N=\frac{1}{\cosh t}JZ$ is the unit normal of $\MM$ and $N$ is time-like, \ie $G(N,N)=-1$.

The structure tensors $\f$ and $\eta$ are determined by the following condition for any tangent vector field $X$ on $\MM$
\[
JX = \f X +\eta(X) J\xi.
\]

The B-metric $g$ is the restriction of $G$ on $\X(\MM)$.

In \cite{GaMiGr}, it is shown that the constructed manifold $\M$ belongs to $\F_5$ because $F=F^5$, where
\begin{equation}\label{ExYS=Lee}
\ta=0,\qquad \ta^*=\frac{2n}{\cosh t}\eta,\qquad \om=0.
\end{equation}

Now, we define the following functions on $\R^{2n+1}=\left\{\left(x^1, \dots, x^n; \A x^{n+1}, \A \dots,\right.\right.$ $\left.\left.\A x^{2n};\A{} t\right)\right\}$
\begin{equation}\label{ExSl=uvw}
\begin{split}
u&=\frac12\sum_{i=1}^{n}\ln \left\{\left(x^i\right)^2+\left(x^{n+i}\right)^2\right\} + \ell(t),\\[4pt]
v&=\sum_{i=1}^n \arctan \frac{x^i}{x^{n+i}},
\\[4pt]
w&=h(t),
\end{split}
\end{equation}
where $\ell(t)$ is an arbitrary twice differentiable function on $\R$ such that $\ell'\neq  0$ and $h(t)$ is an arbitrary differentiable function on $\R$.
Then, we apply an accR transformation determined by these functions $(u,v,w)$.

We obtain the expressions of their partial derivatives for $i\in\{1,2,\dots, n\}$ as follows:
%\begin{equation*}\label{ExSl=d}
%\begin{array}{l}
%\ddu{i}=-\ddv{n+i}=\dfrac{x^i}{(x^i)^2+(x^{n+i})^2},
%\\[9pt]
%\ddu{n+i}=\ddv{i}=\dfrac{x^{n+i}}{(x^i)^2+(x^{n+i})^2},\qquad i\in\{1,2,\dots, n\},
%\\[9pt]
%\dfrac{\p u}{\p t}= \ell', \qquad \dfrac{\p w}{\p t}=h'.
%\end{array}
%\end{equation*}
\begin{equation*}\label{ExSl=d}
\begin{array}{ll}
\ddu{i}=-\ddv{n+i}=\dfrac{x^i}{(x^i)^2+(x^{n+i})^2},
\qquad &
\dfrac{\p u}{\p t}= \ell', \\[9pt]
\ddu{n+i}=\ddv{i}=\dfrac{x^{n+i}}{(x^i)^2+(x^{n+i})^2},
\qquad & \dfrac{\p w}{\p t}=h'.
\end{array}
\end{equation*}
Then we find that the functions from \eqref{ExSl=uvw} have the properties
\begin{equation}\label{exYS=dd}
\D u\circ \f =\D v,\qquad \D u(\xi)=\ell',\qquad \D v(\xi)=0,\qquad \D w=h'\eta.
\end{equation}

As a next step, we apply an accR transformation with the functions $(u,v,w)$ defined by \eqref{ExSl=uvw}.

Taking into account \eqref{albt}, \eqref{ttbartt},  \eqref{ExYS=Lee} and \eqref{exYS=dd}, we obtain the corresponding Lee forms, which coincide with the results in \eqref{F0F1-YS} if and only if the following condition is valid
\[
\ell'=-\frac{1}{\cosh t}.
\]
A solution of this differential equation  is
\[
\ell=-\arctan(\sinh t).
\]
Then, the Lee forms of the obtained manifold are determined by
\begin{equation*}\label{ttbartt0=ex}
    \bar{\ta} = 4n\, \D u\circ \f,
    \qquad
    \bar{\ta}^* = -4n\, \D v\circ \f,
    \qquad
    \bar{\om}=0.
\end{equation*}

Thus, the transformed manifold $(\MM,\f,\bar\xi,\bar\eta,\bar g)$ is an $\F_1$-manifold with a Yamabe soliton with potential $\bar\xi$ and a constant scalar curvature $\bar\tau=\sm$, according ot \thmref{thm:F0-YS}.

\section*{Conclusions}

Steady-state solutions of geometric flows, in particular Yamabe solitons and the like are still a very current topic in differential geometry. The paper reveals new properties of these types of solitons in a contact conformally transformed manifold of an understudied type having a vertical torse-forming potential by considering the interrelation between the original and the deformed structure. Since Yamabe solitons are not yet sufficiently studied, any contribution in this direction may bring new perspectives on the geometry of the manifold.

%%%%%%%%%%%%%%%%%%%%%%%%%%%%%%%%%%%%%%%%%%
\vspace{6pt}

%\funding{This research was funded by Scientific Research Fund,
%University of Plovdiv Paisii Hilendarski, Bulgaria, grant numbers MU21-FMI-008 and FP21-FMI-002.}

% If authors have biography, please use the format below
%\section*{Short Biography of Authors}
%\bio
%{\raisebox{-0.35cm}{\includegraphics[width=3.5cm,height=5.3cm,clip,keepaspectratio]{Definitions/author1.pdf}}}
%{\textbf{Firstname Lastname} Biography of first author}
%
%\bio
%{\raisebox{-0.35cm}{\includegraphics[width=3.5cm,height=5.3cm,clip,keepaspectratio]{Definitions/author2.jpg}}}
%{\textbf{Firstname Lastname} Biography of second author}

% For the MDPI journals use author-date citation, please follow the formatting guidelines on http://www.mdpi.com/authors/references
% To cite two works by the same author: \citeauthor{ref-journal-1a} (\citeyear{ref-journal-1a}, \citeyear{ref-journal-1b}). This produces: Whittaker (1967, 1975)
% To cite two works by the same author with specific pages: \citeauthor{ref-journal-3a} (\citeyear{ref-journal-3a}, p. 328; \citeyear{ref-journal-3b}, p.475). This produces: Wong (1999, p. 328; 2000, p. 475)

%%%%%%%%%%%%%%%%%%%%%%%%%%%%%%%%%%%%%%%%%%
%% for journal Sci
%\reviewreports{\\
%Reviewer 1 comments and authors’ response\\
%Reviewer 2 comments and authors’ response\\
%Reviewer 3 comments and authors’ response
%}
%%%%%%%%%%%%%%%%%%%%%%%%%%%%%%%%%%%%%%%%%%

\end{document}